\documentclass[a4paper,11pt]{amsart}
\usepackage{amsmath,graphicx,amscd,mathrsfs}
\usepackage{epstopdf,amsthm,amssymb,color}

\theoremstyle{plain}
\newtheorem{theorem}{Theorem}[section]

\newtheorem{lemma}[theorem]{Lemma}

\theoremstyle{remark}

\newtheorem*{question}{Question}

\begin{document}

\title{Torus knots obtained by negatively twisting torus knots}

\author[S. Lee]{Sangyop Lee}
\address{Department of Mathematics, Chung-Ang University,
84 Heukseok-ro, Dongjak-gu, Seoul 06974, Korea}
\email{sylee@cau.ac.kr}

\author{Thiago de Paiva}
\address[]{School of Mathematics, Monash University, VIC 3800, Australia }
\email[]{thiago.depaivasouza@monash.edu}

\thanks{The first author was supported by the National Research Foundation of Korea Grant funded by the Korean Government (NRF-2020R1F1A1A01074716).}

\begin{abstract}
Twisted torus knots are torus knots with some full twists added along some number of adjacent strands. There are infinitely many known examples of twisted torus knots which are actually torus knots. We give eight more infinite families of such twisted torus knots with a single negative twist.
\end{abstract}

\maketitle

\section{introduction}

Dean introduced twisted torus knots in his doctoral thesis \cite{Thesis} to study Seifert fibered spaces obtained by Dehn fillings. It has turned out that twisted torus knots have many interesting properties. Many hyperbolic knots in $S^3$ whose complements can be triangulated by a small number of ideal tetrahedra are found among twisted torus knots \cite{simplest, nextsimplest}.
Their volumes ~\cite{generalizedtwistedtoruslinks}, knot Floer homology \cite{homology}, bridge spectra ~\cite{Bridge}, and Heegaard splittings~\cite{Heegaard} have been studied.

Twisted torus knots are described by using four integer parameters and constructed from torus knots as follows. Let $p, q$ be coprime integers with $p > q \ge 1$. Consider a torus knot $T(p, q)$, which is embedded on an unknotted torus $F$ in the $3$-sphere $S^3$. Let $D$ be a disk intersecting $F$ in an arc with its boundary circle surrounding $r$ adjacent strands of $T(p, q)$, where $p + q \geq r \geq 1$. Let $s$ be a non-zero integer. After $(-1/s)$-surgery on the boundary circle of $D$, the $3$-sphere $S^3$ becomes $S^3$ again and the torus knot $T(p,q)$ becomes a new knot. This new knot is called a {\em twisted torus knot} $T(p, q, r, s)$. Here, the surgery coefficients are given in the usual way (see \cite{Rolfsen}).

The classification of the geometry of twisted torus knots has received special attention, mainly by Lee, the first author. He first determined the knot types of twisted torus knots $T(p, q, r, s)$ when $r$ is a multiple of $q$ by showing that $T(p, q, kq, s)$ is the $(q, p + k^2qs)$-cable knot on the torus knot $T(k, ks + 1)$ \cite{cable}. Then he determined the parameters $(p,q,r,s)$ for which $T(p,q,r,s)$ is a trivial knot \cite{unknotted}. For $| s | \geq 2$, under the assumption that $r$ is not a multiple of $q$, he proved that if $(p, q, r, s) = (2n \pm 1, n, n \pm 1, -2)$ for some positive integer $n$, then $T(p, q, r, s)$ is the torus knot $T(2n \pm 1, \mp 2)$~\cite{LeeTorusknotsobtained} and otherwise $T(p, q, r, s)$ is a hyperbolic knot \cite{hyperbolicity}. Therefore, the geometric types of twisted torus knots $T(p, q, r, s)$ with $| s | \geq 2$  are already known. However, the case  $| s | = 1$  has not yet been solved.

For $s = 1$, Lee determined twisted torus knots which are torus knots \cite{Positively} and Paiva, the second author, found an infinite family of satellite twisted torus knots \cite{P}.

In the remainder of the paper, we assume $s=-1$ and focus on the case that $T(p,q,r,s)$ becomes a torus knot. Guntel found the first family of such knots by showing that the twisted torus knots
$T((k +1)q - 1, q, q - 1, - 1)$ are the torus knots $T(kq + 1, q)$ where $q \geq 3$ and
$k \geq 2$ \cite{Guntel}. For twisted torus knots $T(p,q,r,-1)$ whose parameters $(p,q,r)$ are of the form $(p, q, p - kq)$ \cite{torusTwistedtorusknots} or satisfy $(q<)p<r\le p+q$ \cite{L}, Lee determined which of them are torus knots. In this paper, we find eight new infinite families of such twisted torus knots.

\begin{theorem} \label{t:main}
Let $m$ and $n$ be positive integers. Then we have the following:
  \begin{itemize}
    \item [(1)] $T(mn+m+1,mn+1,mn,-1)=T(mn+n+1,m+1)$;
    \item [(2)] $T(mn+m+1,mn+1,mn+m,-1)=T(mn+m-n,-m+1)$;
    \item [(3)] $T(mn+m+1,mn+1,mn+2,-1)=T(mn-n+1,m-1)$;
    \item [(4)] $T(mn+m-1,mn-1,mn+m-2,-1)=T(mn+m-n-2,-m+1)$;
    \item [(5)] $T(mn+m-1,mn-1,mn,-1)=T(mn-n-1,m-1)$;
    \item [(6)] $T(2n+1,n,2n-1,-1)=T(2n-3,-n+1)$;
    \item [(7)] $T(3n-1,n,n+1,-1)=T(2n-1,n-1)$; and
    \item [(8)] $T(3n+1,n,3n-1,-1)=T(3n-2,-2n+1)$.
  \end{itemize}
Here, we assume that $mn\ge 2$ for (4) and (5).
\end{theorem}

For $q<p\leq 30$ and $r\le 29$, we used SnapPy to verify that if $T(p, q, r, -1)$ is a torus knot, then it is one of the knots in \cite[Theorems 1.1 and 1.2]{torusTwistedtorusknots}, \cite[Theorem 1.1]{L} and Theorem \ref{t:main} in this paper. We raise the following question.

\begin{question} Are there any other twisted torus knots $T(p, q,r, s)$ which are torus knots but not listed in \cite[Theorem 1.1 or Theorem 1.2]{torusTwistedtorusknots}, \cite[Theorem 1.1]{L} or Theorem \ref{t:main} in this paper?
\end{question}

\subsection{Acknowledgment} We appreciate the comments of Professor Jessica Purcell, and the second author is grateful to the Faculty of Science, Monash University, for the financial support.

\subsection{Dedication}The author Thiago de Paiva would like to dedicate this paper to the memory of his master's supervisor, Professor Roberto Callejas Bedregal. Roberto died of complications caused by COVID-19 during the writing of this paper. Roberto was an outstanding professor at the Federal University of Para\'iba. He had great influence at the beginning of the second author's career. He would like to thank him. He will be missed.

\section{Braids}

In this section, we prepare some braid isotopies. For this, we first simplify braid diagrams in the following way. By assigning a nonnegative integer $j$ to a single strand, we mean $j$ parallel strands without any twists. For positive integers $k$ and $\ell$, let $(k,\ell)$ denote the $(k,\ell)$-torus braid, i.e., $(k,\ell)$ is the braid $(\sigma_1\sigma_2\cdots \sigma_{k-1})^\ell$, where $\sigma_i$ is an elementary braid which is obtained from the trivial braid on $k$ strands by letting the $i$th strand cross under the $(i+1)$st strand (see \cite[Figure 2]{BB}). Let $\overline{(k,\ell)}$ denote the braid $(\sigma_{k-1}\sigma_{k-2}\cdots \sigma_1)^\ell$. Let $(k,-\ell)$ and $\overline{(k,-\ell)}$ denote the mirror images of $\overline{(k,\ell)}$ and $(k,\ell)$, respectively. Here, by the mirror image of a given braid $\beta$, we mean the braid obtained from $\beta$ by changing all crossings. 
Note that $(k,-\ell)$ is the braid $(\sigma_{k-1}^{-1}\sigma_{k-2}^{-1}\cdots\sigma_1^{-1})^\ell$ and $\overline{(k,-\ell)}$ is the braid $(\sigma_1^{-1}\sigma_2^{-1}\cdots \sigma_{k-1}^{-1})^\ell$. Let $\ell_k$ denote $\ell$ full twists on $k$ strands. Then $\ell_k=(k,\ell k)$ and $\overline{\ell_k}=\overline{(k,\ell k)}$. Note that $\overline{\ell_k}=\overline{(k,\ell k)}$ is the mirror image of $-\ell_k$. See Figure \ref{f:a}. It is easy to see that if $\ell_1$ and $\ell_2$ are integers, then $(k,\ell_1)\cdot(k,\ell_2)=(k,\ell_1+\ell_2)$ and $\overline{(k,\ell_1)}\cdot\overline{(k,\ell_2)}=\overline{(k,\ell_1+\ell_2)}$, where $\beta_1\cdot\beta_2$ is the braid obtained by stacking the braid $\beta_1$ on top of the braid $\beta_2$.

\begin{figure}[tbh]
\includegraphics{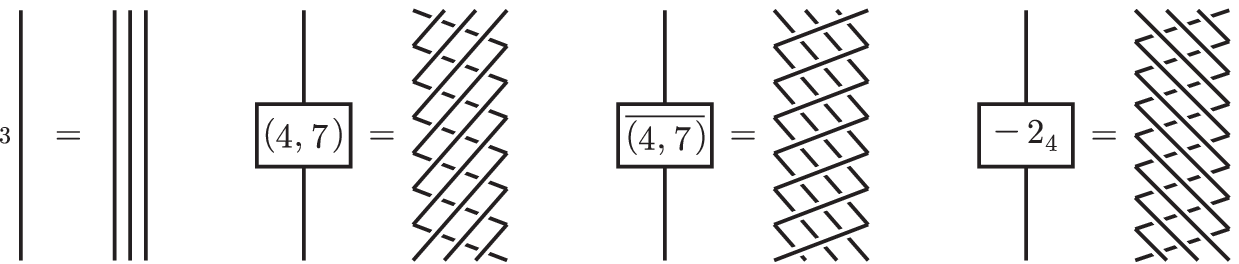}\caption{Braids} \label{f:a}
\end{figure}

\begin{lemma}\label{l:fst}
Let $k$ be a positive integer. Then the three braids in Figure \ref{f:e} are isotopic.
\end{lemma}
\begin{figure}[tbh]
\includegraphics{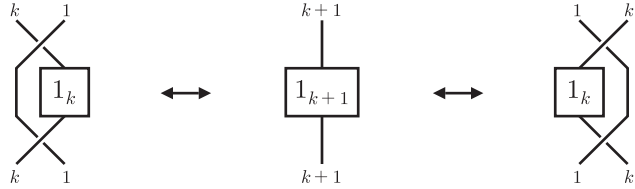}\caption{Isotopic braids} \label{f:e}
\end{figure}
\begin{proof}
This follows immediately from \cite[Lemma 2.3]{L}.
\end{proof}

\begin{lemma}\label{l:mrm}
Let $k,\ell$ be positive integers. Then the mirror image of the braid $\ell_k$ is isotopic to the braid $-\ell_k$.
\end{lemma}
\begin{proof}
It is enough to prove that the mirror image of $-1_k$, which is $\overline{1_k}$, is isotopic to $1_k$. We prove this by an induction on $k$. When $k=1$ or $2$, this is obviously true. Suppose that $\overline{1_k}$ is isotopic to $1_k$. One easily sees that $\overline{1_{k+1}}$ is isotopic to the braid on the left of Figure \ref{f:e} with $1_k$ replaced by $\overline{1_k}$. It follows from the induction hypothesis and Lemma \ref{l:fst} that $\overline{1_{k+1}}$ is isotopic to $1_{k+1}$.
\end{proof}

\begin{lemma}\label{l:twb}
Let $k$ be a positive integer. Then the two braids in Figure \ref{f:g}(x) are the same for each $x=a,b$.
\end{lemma}
\begin{figure}[tbh]
\includegraphics{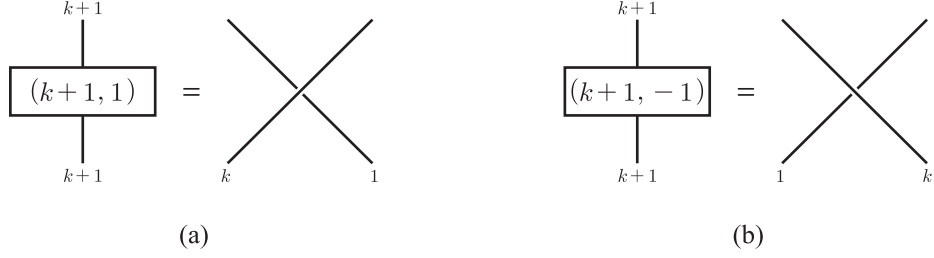}\caption{Braids $(k+1,1)$ and $(k+1,-1)$} \label{f:g}
\end{figure}
\begin{proof}
This follows immediately from the definitions at the beginning of this section.
\end{proof}

\begin{lemma}\label{l:dst}
Let $\beta,\beta'$ be braids on $j+k$ strands, where $j,k$ are nonnegative integers. Then the links obtained by closing the braids in Figure \ref{f:d}(x) are equivalent for each $x=a,b,c,d$.
\end{lemma}
\begin{figure}[tbh]
\includegraphics{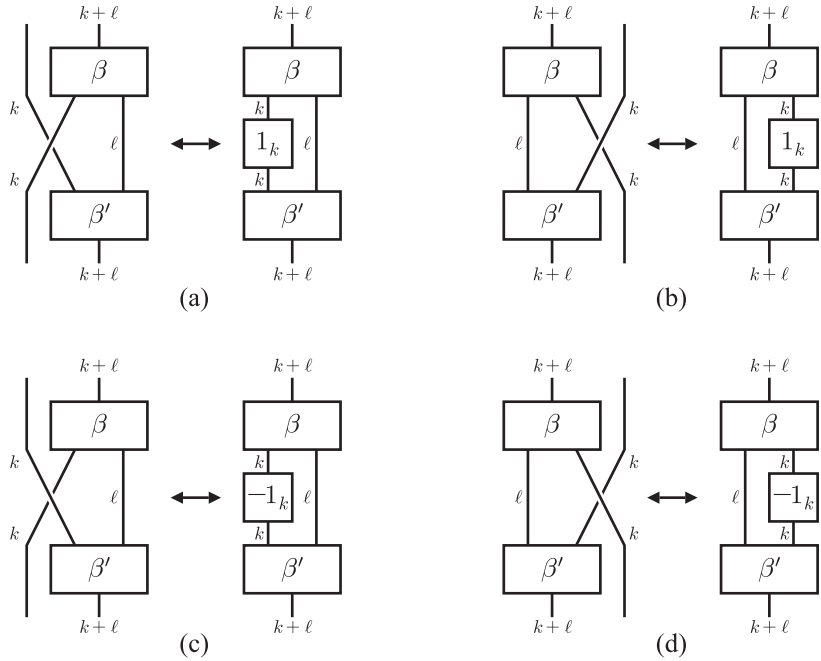}\caption{Generalized destabilization/stabilization} \label{f:d}
\end{figure}
\begin{proof}
This is \cite[Lemma 2.5]{L}.
\end{proof}
For each $x=a,b,c,d$, we call the move from the left braid in Figure \ref{f:d}(x) to the right a {\em generalized destabilization} and the move from the right to the left a {\em generalized stabilization}.

\begin{lemma}\label{l:cjg}
Let $\beta_1,\beta_2$ be braids on $k$ strands. Then the closures of $\beta_1\cdot\beta_2$ and $\beta_2\cdot\beta_1$ are equivalent knots or links.
\end{lemma}
\begin{proof}
It is well known that conjugate braids yield equivalent knots or links. Also, for two group elements $a$ and $b$, $ab$ is conjugate to $ba(=a^{-1}(ab)a)$.
\end{proof}

\begin{lemma}\label{l:tbt}
Let $j,k,\ell$ be positive integers with $k\le \ell<j$. Then the braid in the center of Figure \ref{f:b}(x) is isotopic to any of the upper braids in the figure for each $x=a,b$. In particular, if $k=\ell$, then the central braid is isotopic to any of the lower braids. Also, if $k=1$, then the five braids in Figure \ref{f:b}(y) are isotopic for each $y=c,d$.
\end{lemma}
\begin{figure}[tbh]
\includegraphics{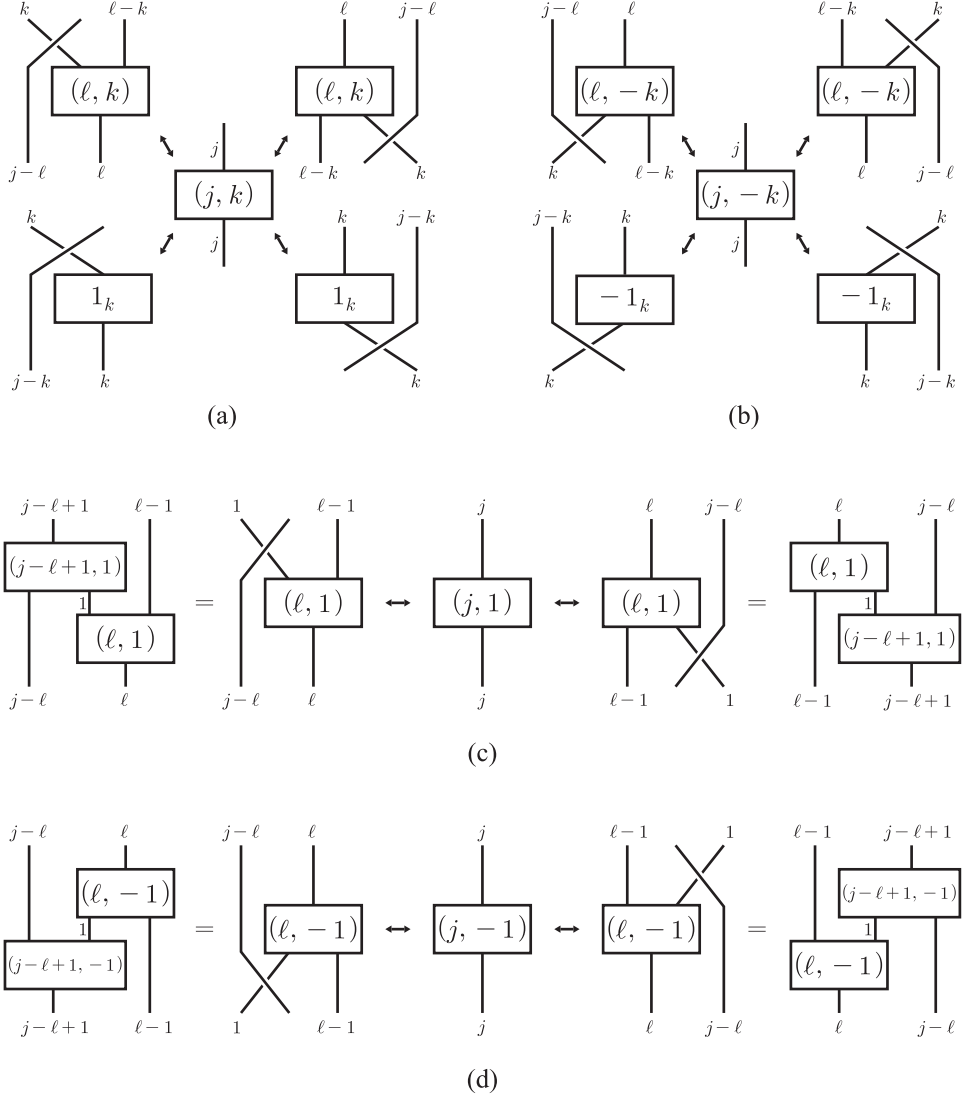}\caption{Braid isotopies} \label{f:b}
\end{figure}
\begin{proof}
The central and two upper braids in Figure \ref{f:b}(a) are isotopic as shown in Figure \ref{f:c}, which illustrates the case that $(j,k,\ell)=(9,3,5)$. Similarly for the braids in Figure \ref{f:b}(b). The last two statements of the lemma follow immediately from the general case. The first/last two braids in Figure \ref{f:b}(c),(d) are the same by Lemma \ref{l:twb}.
\begin{figure}[tbh]
\includegraphics{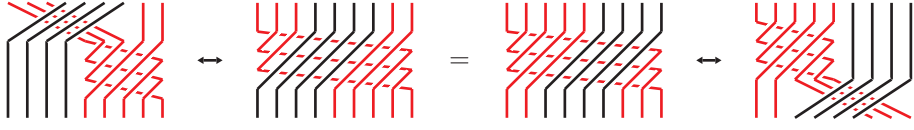}\caption{Braid isotopies} \label{f:c}
\end{figure}
\end{proof}

\begin{lemma}\label{l:pqr}
Let $p,q,r$ be positive integers such that $p\ge r\ge q,r+q\ge p$, and $p,q$ are coprime. Then the following hold.
\begin{itemize}
  \item [(1)] If $2q>p$, then the twisted torus knot $T(p,q,r,-1)$ is obtained by closing the braid in Figure \ref{f:f}(i).
  \item [(2)] If $p\ge 2q\ge r$, then the twisted torus knot $T(p,q,r,-1)$ is obtained by closing the braid in Figure \ref{f:f}(l).
\end{itemize}
\end{lemma}
\begin{figure}[tbh]
\includegraphics{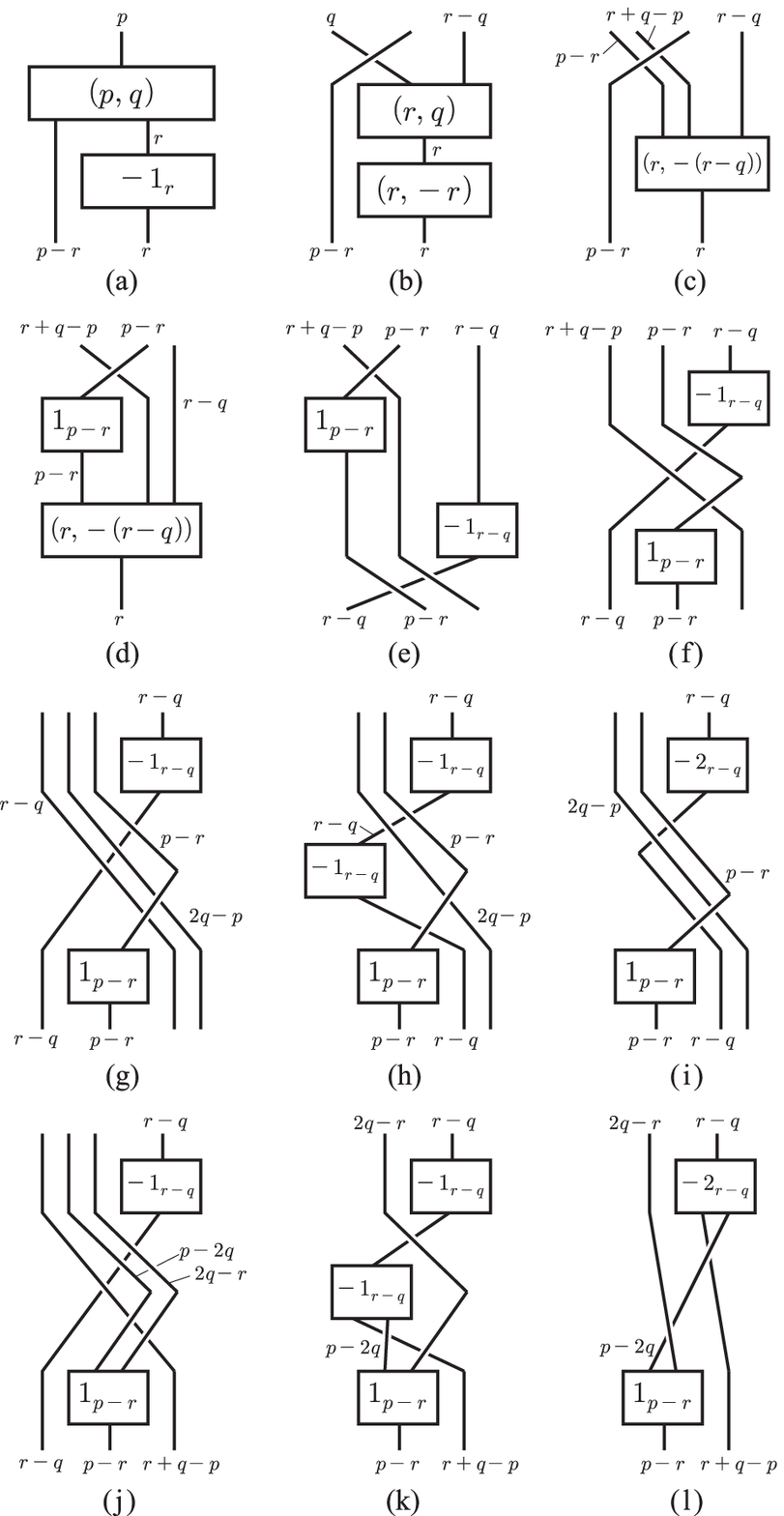}\caption{Braids} \label{f:f}
\end{figure}
\begin{proof}
The twisted torus knot $T(p,q,r,-1)$ can be obtained by closing the braid in Figure \ref{f:f}(a). One can see that the closures of the braids in Figure \ref{f:f}(a)$\sim$(f) are equivalent knots as follows:
\begin{itemize}
  \item From (a) to (b): Apply the upper left isotopy in Figure \ref{f:b}(a) with letting $(j,k,\ell)=(p,q,r)$.
  \item From (b) to (c): Split the family of $q$ parallel strands into two families, one containing $p-r$ parallel strands and the other containing $r+q-p$ parallel strands. Note that $p-r\ge 0$ and $r+q-p\ge 0$. Combine the two torus braids $(r,q)$ and $(r,-r)$ to obtain $(r,-(r-q))$.
  \item From (c) to (d): Apply a generalized destabilization.
  \item From (d) to (e): Apply the lower left isotopy in Figure \ref{f:b}(b) with letting $(j,k)=(r,r-q)$.
  \item From (e) to (f): Pull down the full twist $1_{p-r}$ and then apply a third Reidemeister move.
\end{itemize}

Suppose $2q>p$. One can see that the closures of the braids in Figure \ref{f:f}(f)$\sim$(i) are equivalent knots as follows:
\begin{itemize}
  \item From (f) to (g): Split the family of $r+q-p$ parallel strands into two families, one containing $r-q$ parallel strands and the other $2q-p$ parallel strands. Note that $r-q\ge 0$ and $2q-p\ge 0$.
  \item From (g) to (h): Apply a generalized destabilization.
  \item From (h) to (i): Combine the two negative full twists on $r-q$ strands into $-2_{r-q}$.
  \end{itemize}

Suppose $r\le 2q\le p$. One can see that the closures of the braids in Figure \ref{f:f}(f),(j)$\sim$(l) are equivalent knots as follows:
\begin{itemize}
  \item From (f) to (j): Split the family of $p-r$ parallel strands into two families, one containing $p-2q$ parallel strands and the other $2q-r$ parallel strands. Note that $p-2q\ge 0$ and $2q-r\ge 0$.
  \item From (j) to (k): Note that $(p-2q)+(r+q-p)=r-q$. Apply a generalized destabilization.
  \item From (k) to (l): Combine the two negative full twists on $r-q$ strands into $-2_{r-q}$.
  \end{itemize}

This completes the proof.
\end{proof}

\begin{lemma}\label{l:idc}
Let $m,n$ be positive integers. Let $\beta_k$ denote the braid in Figure \ref{f:h}(a), where $k$ is an integer with $0\le k\le n$. Then the closure of $\beta_k$ is the torus knot $T(mn+n+1,m+1)$ for any $k$.
\end{lemma}
\begin{figure}[tbh]
\includegraphics{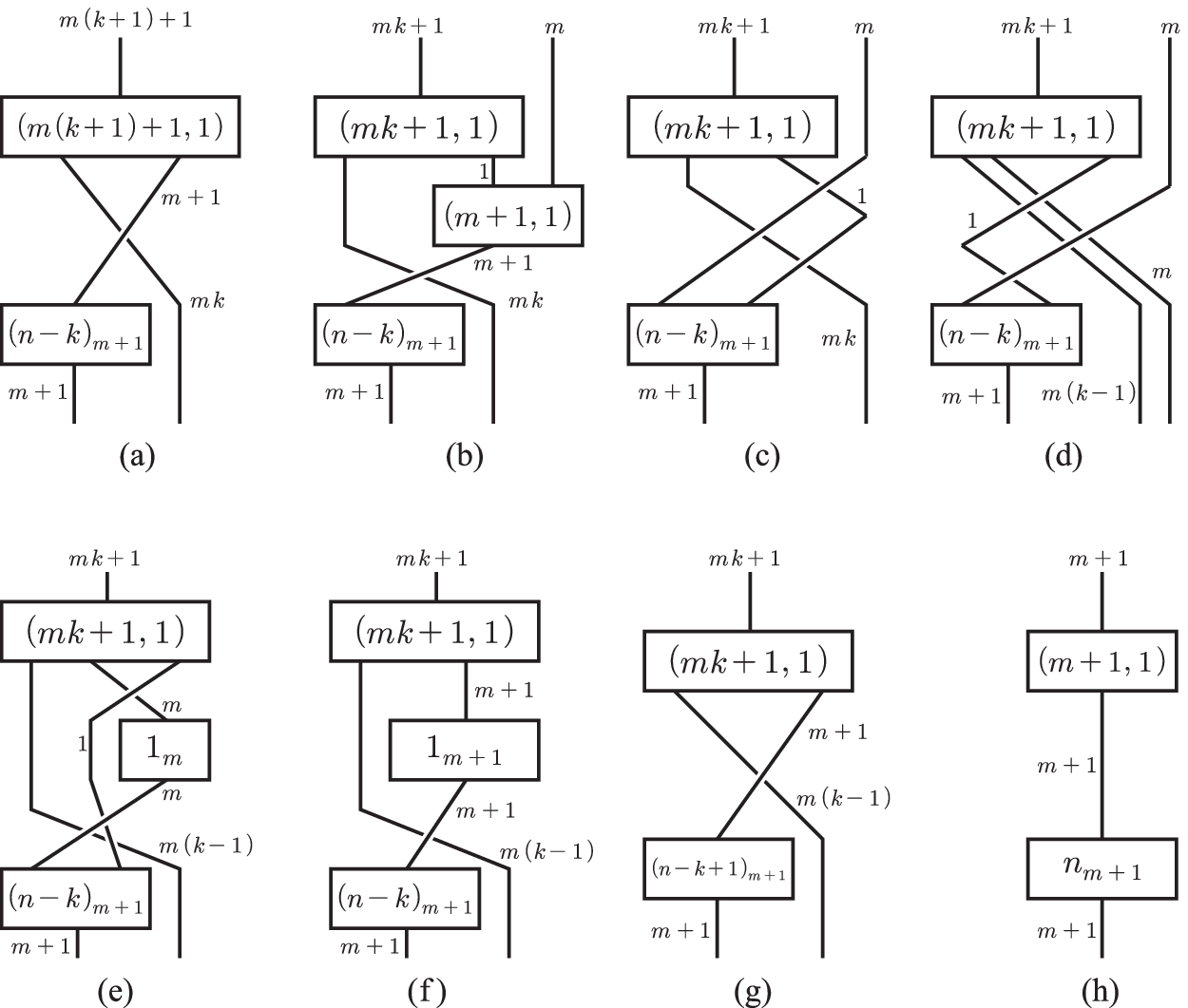}\caption{Braids} \label{f:h}
\end{figure}
\begin{proof}
One can see that the closures of the braids in Figure \ref{f:h}(a)$\sim$(g) are equivalent knots as follows:
\begin{itemize}
  \item From (a) to (b): Apply the left isotopy in Figure \ref{f:b}(c) with letting $(j,\ell)=(m(k+1)+1,m+1)$.
  \item From (b) to (c): Apply Lemma \ref{l:twb}.
  \item From (c) to (d): Apply a third Reidemeister move and split the family of $mk$ parallel strands into two families, one containing $m$ parallel strands and the other $m(k-1)$ parallel strands.
  \item From (d) to (e): Apply a generalized destabilization and a third Reidemeister move.
  \item From (e) to (f): Apply Lemma \ref{l:fst}.
  \item From (f) to (g): Combine the braids $1_{m+1}$ and $(n-k)_{m+1}$ to obtain $(n-k+1)_{m+1}$.
\end{itemize}
Noting that the braid in Figure \ref{f:h}(g) is $\beta_{k-1}$, one sees that the closures of braids $\beta_k(k=0,1,\ldots,n)$ are the same knot. In particular, it is clear that $\beta_0$ is the braid in Figure \ref{f:h}(h) and its closure is the torus knot $T(mn+n+1,m+1)$.
\end{proof}

\begin{lemma}\label{l:cun}
Let $m,n,\varepsilon$ be integers such that $m,n$ are positive and $\varepsilon=\pm 1$. Let $\alpha$ be a braid on $m-1$ strands. Let $\beta(\varepsilon,\alpha)$ and $\gamma(\varepsilon,\alpha)$ denote the braids in Figure \ref{f:l}(a) and (b), respectively. Then the closures of these braids are equivalent knots or links.
\end{lemma}
\begin{figure}[tbh]
\includegraphics{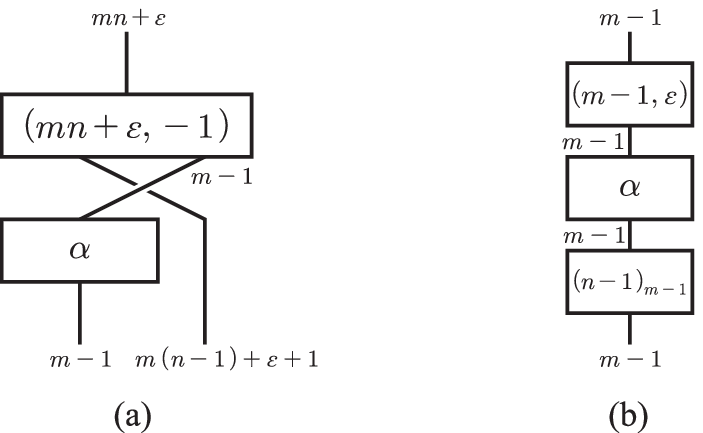}\caption{Braids} \label{f:l}
\end{figure}
\begin{proof}
For an integer $k(1\le k\le n)$, let $\beta_k(\varepsilon,\alpha)$ denote the braid in Figure \ref{f:i}(a). In particular, $\beta_n(\varepsilon,\alpha)=\beta(\varepsilon,\alpha)$.
\begin{figure}[tbh]
\includegraphics{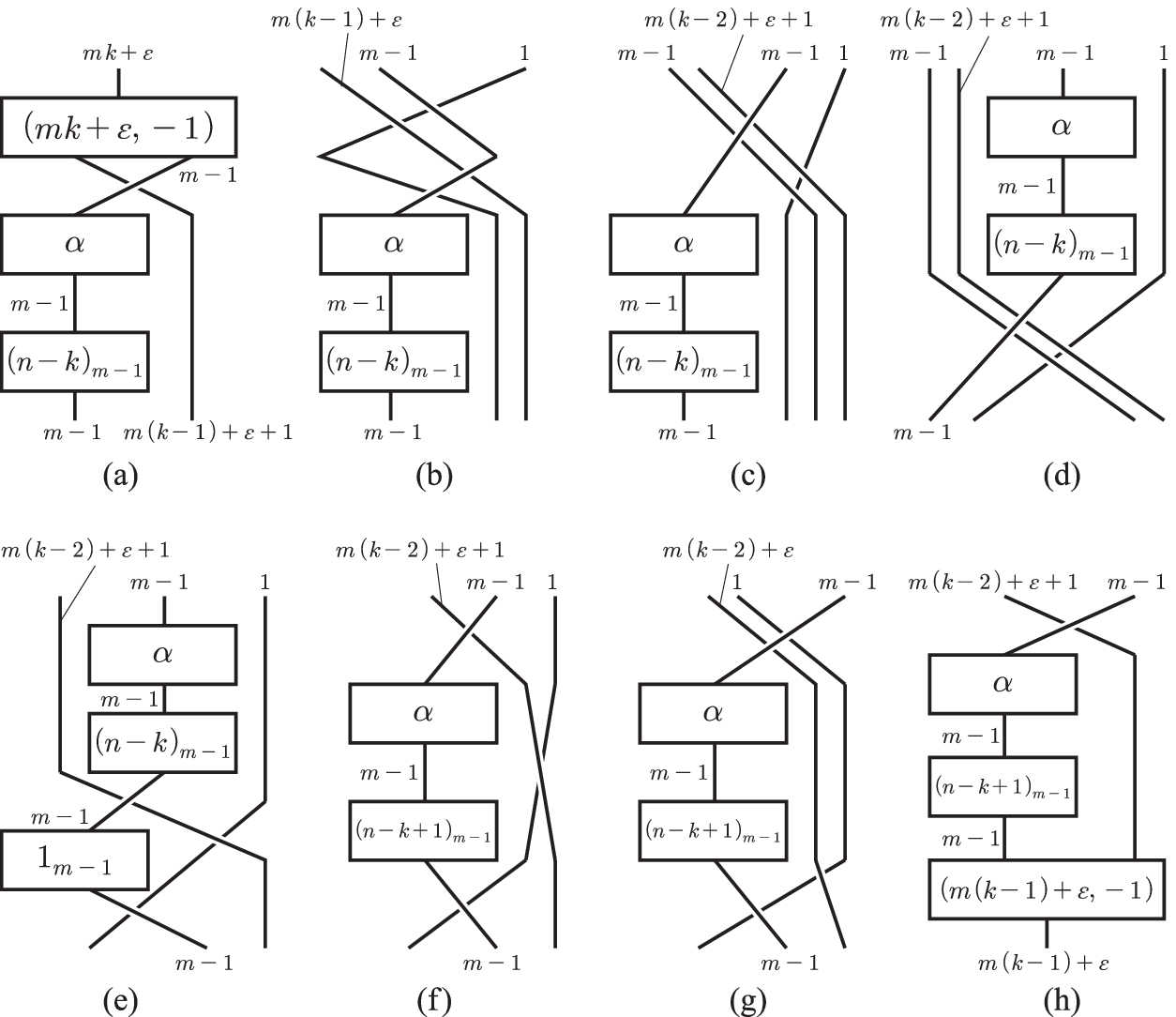}\caption{Braids} \label{f:i}
\end{figure}

One can see that the closures of the braids in Figure \ref{f:h}(a)$\sim$(g) are equivalent knots or links as follows:
\begin{itemize}
  \item From (a) to (b): Apply Lemma \ref{l:twb}.
  \item From (b) to (c): Apply a third Reidemeister move and then a second Reidemeister move.  Split the family of $m(k-1)+\varepsilon$ parallel strands into two families, one containing $m-1$ parallel strands and the other $m(k-2)+\varepsilon+1$ parallel strands.
  \item From (c) to (d): Pull up the braids $\alpha$ and $(n-k)_{m-1}$.
  \item From (d) to (e): Apply a generalized destabilization.
  \item From (e) to (f): Pull down the braids $\alpha$ and $(n-k)_{m-1}$ and then combine the braids $(n-k)_{m-1}$ and $1_{m-1}$ to obtain $(n-k+1)_{m-1}$.
  \item From (f) to (g): Destabilize the braid in Figure \ref{f:i}(f).
  \item From (g) to (h): Apply Lemma \ref{l:twb}.
\end{itemize}
One sees that the closure of the braid in Figure \ref{f:i}(h) is equivalent to that of $\beta_{k-1}(\varepsilon,\alpha)$ by Lemma \ref{l:cjg}. Thus the closures of the braids $\beta_k(\varepsilon,\alpha)$ are equivalent knots or links for all $k=1,\ldots, n$.

Consider the braid $\beta_1(\varepsilon,\alpha)$. If $\varepsilon=-1$, then one easily sees that $\beta_1(\varepsilon,\alpha)=\gamma(\varepsilon,\alpha).$ Suppose $\varepsilon=1$. Then $\beta_1(\varepsilon,\alpha)$ is the braid in Figure \ref{f:k}(a). One can see that the closures of the braids in Figure \ref{f:k}(a)$\sim$(e) are equivalent knots or links as follows:
\begin{itemize}
  \item From (a) to (b): Apply Lemma \ref{l:twb}.
  \item From (b) to (c): Pull to the left a family of $m-1$ parallel strands.
  \item From (c) to (d): Destabilize and split the family of $m-1$ parallel strands into a family of $m-2$ parallel strands and a single strand.
  \item From (d) to (e): Destabilize.
\end{itemize}
\begin{figure}[tbh]
\includegraphics{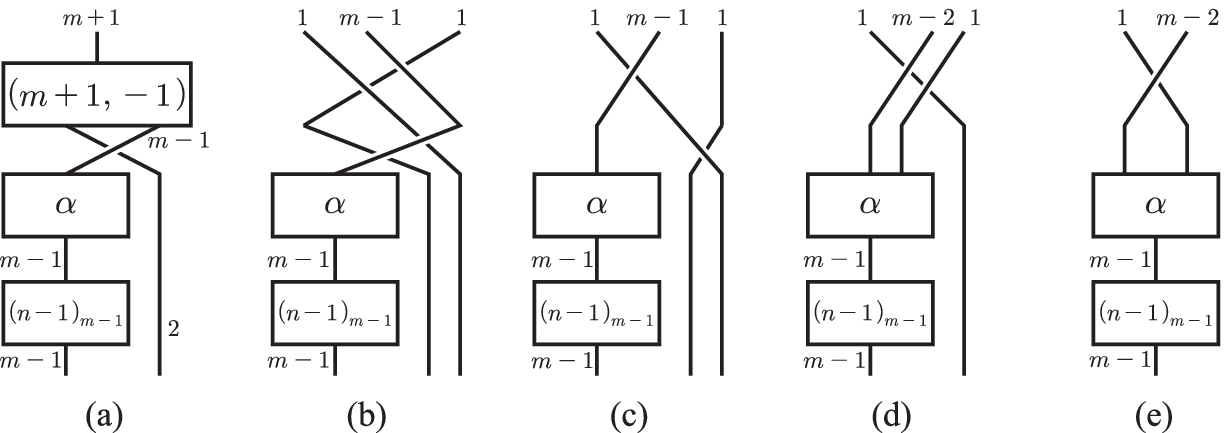}\caption{Braids} \label{f:k}
\end{figure}
The braid in Figure \ref{f:k}(e) is the braid $\gamma(\varepsilon,\alpha)$ with $\varepsilon=1$ by Lemma \ref{l:twb}.
\end{proof}

\section{Proof of Theorem \ref{t:main}}
In this section, we prove Theorem \ref{t:main}.

(1) Consider the twisted torus knot $T(mn+m+1, mn+1, mn, -1)$. It is the closure of the braid in Figure \ref{f:j}(a). The torus braid $(mn+m+1, mn+1)$ splits into two torus braids $(mn+m+1, 1)$ and $(mn+m+1, mn)$ as shown in Figure \ref{f:j}(b). The lower left isotopy in Figure \ref{f:b}(a) with $(j,k)=(mn+m+1,mn)$ yields the braid in Figure \ref{f:j}(c). The braids $1_{mn}$ and $-1_{mn}$ are merged into a trivial braid on $mn$ strands, so we get the braid in Figure \ref{f:j}(d), which is $\beta_n$ in Lemma \ref{l:idc}. Hence $T(mn+m+1, mn+1, mn, -1)=T(mn+n+1,m+1)$ by Lemma \ref{l:idc}.

\begin{figure}[tbh]
\includegraphics{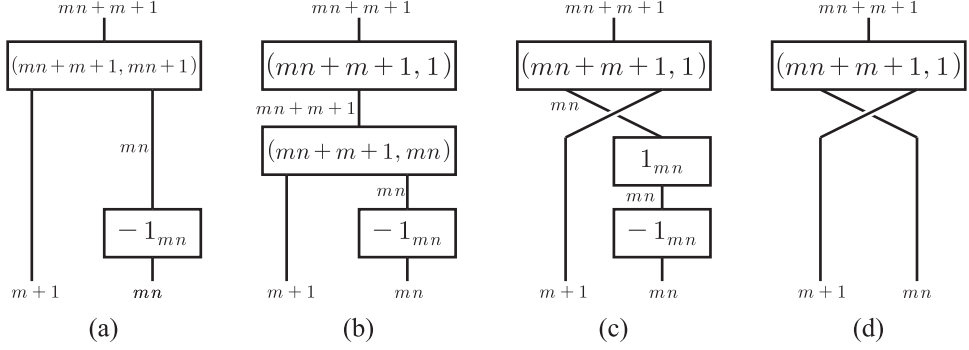}\caption{Braid isotopy} \label{f:j}
\end{figure}

(2) Let $p=mn+m+1, q=mn+1, r=mn+m$.  Then $p\ge r\ge q, r+q\ge p$ and $2q>p$. By Lemma \ref{l:pqr}(1) $T(mn+m+1, mn+1, mn+m, -1)$ is obtained by closing the braid in Figure \ref{f:f}(i).
Note that
\begin{equation*}
\begin{split}
r-q&=m-1,\\
p-r&=1,\\
(2q-p)+(p-r)&=m(n-1)+2,\\
(r-q)+(2q-p)&=mn.
\end{split}
\end{equation*}
Hence by using Lemma \ref{l:mrm}, one can see that the braid in Figure \ref{f:f}(i) is the mirror image of the braid $\beta(1,2_{m-1})$ in Lemma \ref{l:cun} after a $\pi$-rotation. It is easy to see that the braid $\gamma(1,2_{m-1})$ in the lemma is closed to be the torus knot $T(m-1, mn+m-n)$. Thus $T(mn+m+1, mn+1, mn+m, -1)$ is the mirror image of $T(m-1, mn+m-n)$ by Lemma \ref{l:cun}, i.e., $T(mn+m+1, mn+1, mn+m, -1)=T(mn+m-n,-m+1)$.\\

(3) Let $p=mn+m+1,q=mn+1,r=mn+2$. Then $p\ge r\ge q, r+q\ge p$ and $2q>p$. By Lemma \ref{l:pqr}(1) $T(mn+m+1, mn+1, mn+2, -1)$ is obtained by closing the braid in Figure \ref{f:f}(i).
Note that
\begin{equation*}
\begin{split}
r-q&=1,\\
p-r&=m-1,\\
(2q-p)+(p-r)&=mn,\\
(r-q)+(2q-p)&=m(n-1)+2.
\end{split}
\end{equation*}
Hence one can see that the braid in Figure \ref{f:f}(i) is the braid $\beta(1,1_{m-1})$ in Lemma \ref{l:cun}. It is easy to see that the braid $\gamma(1,1_{m-1})$ in the lemma is closed to be the torus knot $T(m-1, mn-n+1)$. Thus $T(mn+m+1, mn+1, mn+2, -1)=T(m-1, mn-n+1)=T(mn-n+1,m-1)$ by Lemma \ref{l:cun}.\\

(4) Let $p=mn+m-1,q=mn-1,r=mn+m-2$, where $mn\ge 2$. Then $p\ge r\ge q, r+q\ge p$ and $2q>p$ except when $(m,n)=(1,2)$ or $n=1$: in the former, both of $T(p,q,r,-1)=T(2,1,1,-1)$ and $T(mn+m-n-2,-m+1)=T(-1,0)$ are the unknot, and in the latter, $T(p,q,r,-1)=T(2m-1,m-1,2m-2,-1)$ is the torus knot $T(mn+m-n-2,-m+1)=T(2m-3,-m+1)$ by \cite[Theorem 1]{cable}. By Lemma \ref{l:pqr}(1) $T(mn+m+1, mn+1, mn+2, -1)$ is obtained by closing the braid in Figure \ref{f:f}(i).
Note that
\begin{equation*}
\begin{split}
r-q&=m-1,\\
p-r&=1,\\
(2q-p)+(p-r)&=m(n-1),\\
(r-q)+(2q-p)&=mn-2.
\end{split}
\end{equation*}
Hence by using Lemma \ref{l:mrm}, one can see that the braid in Figure \ref{f:f}(i) is the mirror image of the braid $\beta(-1,2_{m-1})$ in Lemma \ref{l:cun} after a $\pi$-rotation. It is easy to see that the braid $\gamma(-1,2_{m-1})$ in the lemma is closed to be the torus knot $T(m-1, mn+m-n-2)$. Thus $T(mn+m-1,mn-1,mn+m-2,-1)=T(mn+m-n-2,-m+1)$ by Lemma \ref{l:cun}.\\

(5) Let $p=mn+m-1,q=mn-1,r=mn$, where $mn\ge 2$. Then $p\ge r\ge q, r+q\ge p$ and $2q>p$ except when $(m,n)=(1,2)$ or $n=1$: in the former, both of $T(p,q,r,-1)=T(2,1,2,-1)$ and $T(mn-n-1,m-1)=T(-1,0)$ are the unknot, and in the latter, $T(p,q,r,-1)=T(2m-1,m-1,m,-1)$ is the torus knot $T(mn-n-1,m-1)=T(m-2,m-1)$ by \cite[Theorem 1.2( 2-ii)]{torusTwistedtorusknots}(note that $2m-1=(m-2)f_3+f_4,m-1=(m-2)f_1+f_2$, and $m=(m-2)f_2+f_3$, where $f_i$ are Fibonacci numbers with $f_1=f_2=1$). By Lemma \ref{l:pqr}(1) $T(mn+m-1,mn-1,mn,-1)$ is obtained by closing the braid in Figure \ref{f:f}(i).
Note that
\begin{equation*}
\begin{split}
r-q&=1,\\
p-r&=m-1,\\
(2q-p)+(p-r)&=mn-2,\\
(r-q)+(2q-p)&=m(n-1).
\end{split}
\end{equation*}
Hence one can see that the braid in Figure \ref{f:f}(i) is the braid $\beta(-1,1_{m-1})$ in Lemma \ref{l:cun}. It is easy to see that the braid $\gamma(-1,1_{m-1})$ in the lemma is closed to be the torus knot $T(m-1,mn-n-1)$. Thus $T(mn+m-1,mn-1,mn,-1)=T(mn-n-1,m-1)$ by Lemma \ref{l:cun}.\\

(6) Let $p=2n+1,q=n,r=2n-1$. Then $p\ge r\ge q, r+q\ge p$ and $p\ge 2q\ge r$ except when $n=1$: in this case, both of $T(p,q,r,-1)=T(3,1,1,-1)$ and $T(2n-3,-n+1)=T(-1,0)$ are the unknot. By Lemma \ref{l:pqr}(2) $T(2n+1,n,2n-1,-1)$ is obtained by closing the braid in Figure \ref{f:f}(l). This braid is the first braid in Figure \ref{f:m}. We obtain the second braid in the figure after a second Redemeister move and then the third braid after a destabilization. By using \cite[Lemma 2.4]{torusTwistedtorusknots}, Lemma \ref{l:cjg}, and the fact that a full twist commutes with any braid, one easily sees that the closure of the third braid is the torus knot $T(n-1,-2n+3)$. Thus $T(2n+1,n,2n-1,-1)=T(2n-3,-n+1)$.
\begin{figure}[tbh]
\includegraphics{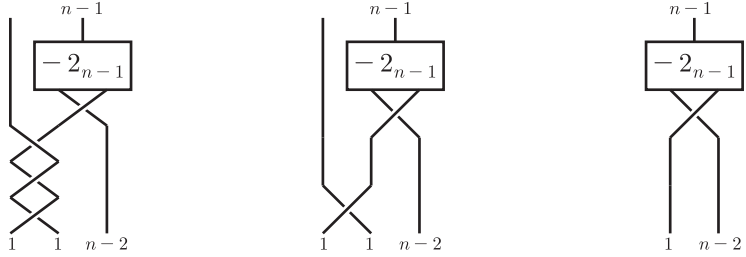}\caption{Braid isotopy} \label{f:m}
\end{figure}\\

(7) Consider the twisted torus knot $T(3n-1,n,n+1,-1)$. If $n=1$, then $T(3n-1,n,n+1,-1)=T(2,1,2,-1)$ and $T(2n-1,n-1)=T(1,0)$ are both the unknot. Hence we may assume $n\ge 2$. The knot $T(3n-1,n,n+1,-1)$ is the closure of the braid in Figure \ref{f:n}(a).
\begin{figure}[tbh]
\includegraphics{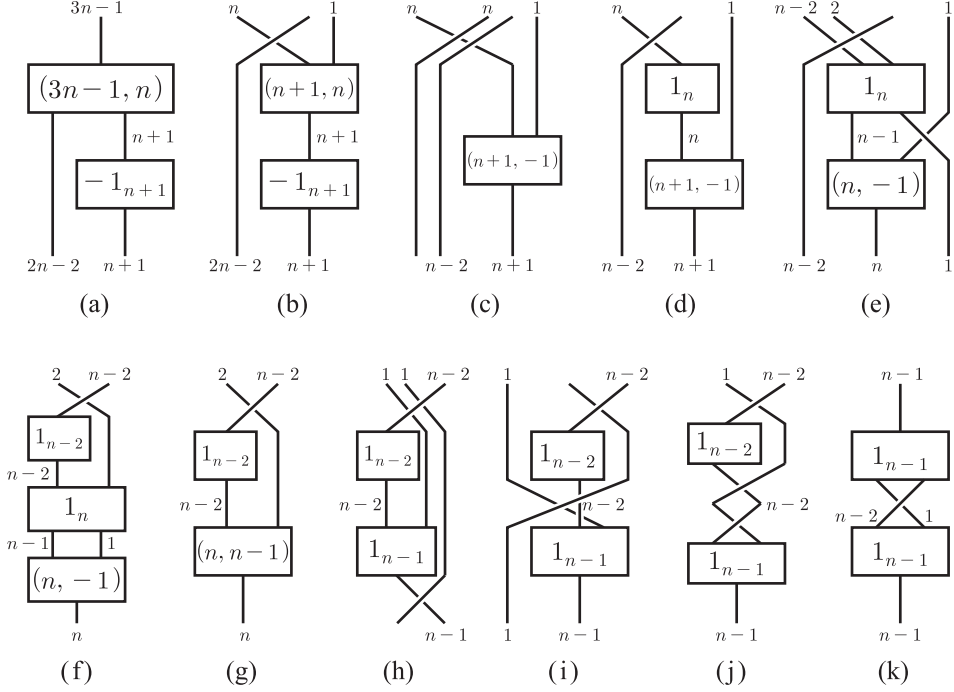}\caption{The closures of these braids are the same knot.} \label{f:n}
\end{figure}
One can see that the closures of the braids in Figure \ref{f:n}(a)$\sim$(k) are equivalent knots as follows:
\begin{itemize}
  \item From (a) to (b): Apply the upper left isotopy in Figure \ref{f:b}(a) with letting $(j,k,\ell)=(3n-1,n,n+1)$.
  \item From (b) to (c): Split the family of $2n-2$ parallel strands into two families, one containing $n$ parallel strands and the other $n-2$ parallel strands.
  \item From (c) to (d): Apply a generalized destabilization.
  \item From (d) to (e): Split a family of $n$ parallel strands into two families, one containing $n-2$ parallel strands and the other $2$ parallel strands. Apply the right isotopy in Figure \ref{f:b}(d) to the braid $(n+1,-1)$ with letting $(j,\ell)=(n+1,n)$.
  \item From (e) to (f): Apply a generalized destabilization to the left of the braid in Figure \ref{f:n}(e) and a destabilization to the right.
  \item From (f) to (g): Combine the braids $1_n$ and $(n,-1)$.
  \item From (g) to (h): Split the family of $2$ parallel strands into two single strands and apply the lower right isotopy in Figure \ref{f:b}(a) to the braid $(n,n-1)$ with letting $(j,k)=(n,n-1)$.
  \item From (h) to (i): Pull the two single strands to the left.
  \item From (i) to (j): Destabilize the braid in Figure \ref{f:n}(i).
  \item From (j) to (k): Use Lemma \ref{l:fst}.
\end{itemize}
One easily sees that the closure of the braid in Figure \ref{f:n}(k) is the torus knot $T(n-1,2n-1)$. Thus $T(3n-1,n,n+1,-1)=T(2n-1,n-1)$.\\

(8) Consider the twisted torus knot $T(3n+1,n,3n-1,-1)$. If $n=1$, then $T(3n+1,n,3n-1,-1)=T(4,1,2,-1)$ and $T(3n-2,-2n+1)=T(1,-1)$ are both the unknot. Hence we may assume $n\ge 2$. The knot $T(3n+1,n,3n-1,-1)$ is the closure of the braid in Figure \ref{f:o}(a).
\begin{figure}[tbh]
\includegraphics{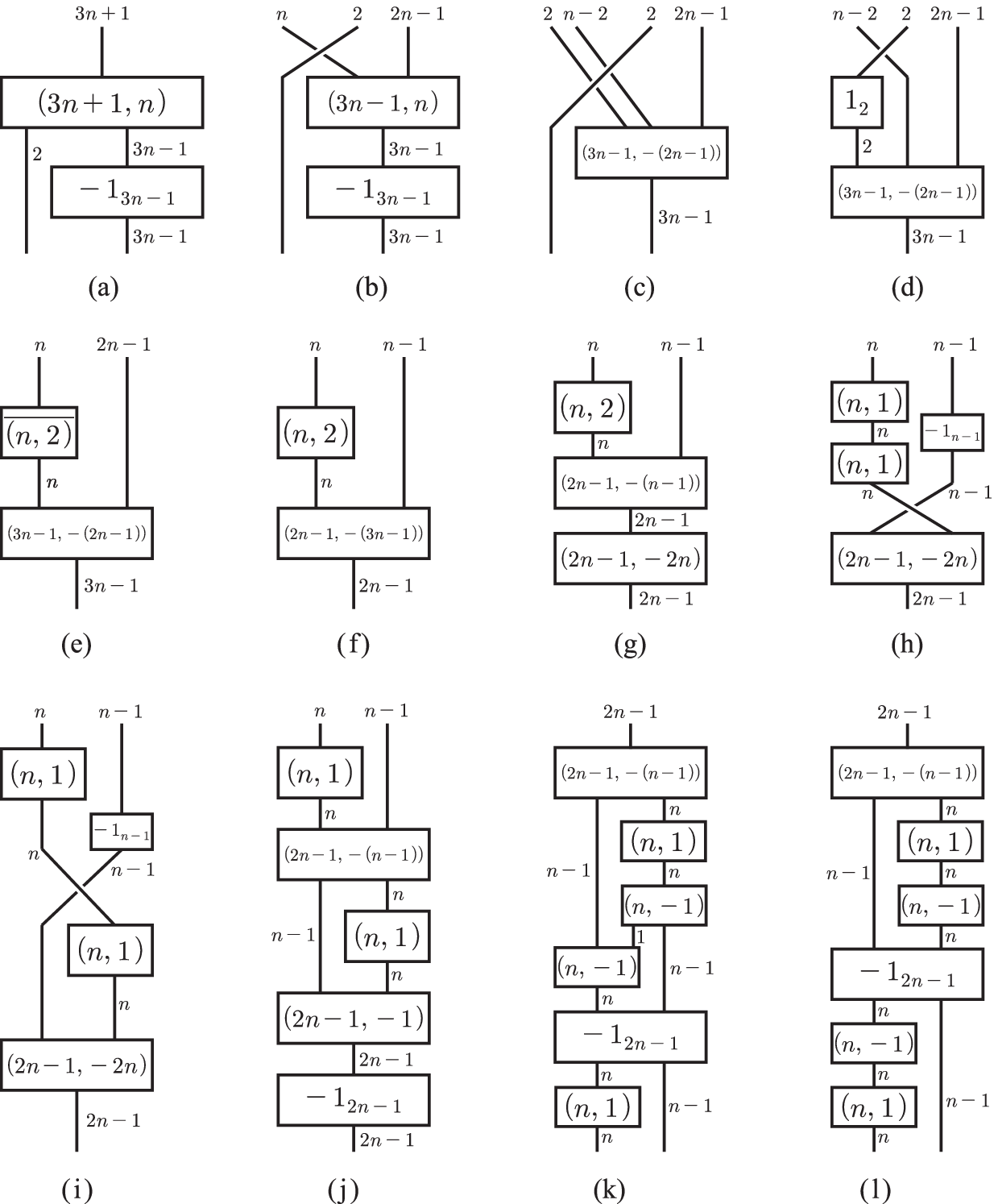}\caption{The closures of these braids are the same knot.} \label{f:o}
\end{figure}
One can see that the closures of the braids in Figure \ref{f:n}(a)$\sim$(k) are equivalent knots as follows:
\begin{itemize}
  \item From (a) to (b): Apply the upper left isotopy in Figure \ref{f:b}(a) with letting $(j,k,\ell)=(3n+1,n,3n-1)$.
  \item From (b) to (c): Combine the braids $(3n-1,n)$ and $-1_{3n-1}$, and split the family of $n$ parallel strands into two families, one containing $2$ parallel strands and the other $n-2$ parallel strands.
  \item From (c) to (d): Apply a generalized destabilization.
  \item From (d) to (e): The braids $\overline{(n,2)}$ and $1_2$ are the mirror images of the braids $(n,-2)$ and $-1_2$, respectively. Thus one can see that the mirror images of the braids in (d) and (e) are isotopic by using the lower right isotopy in Figure \ref{f:b}(b) with letting $(j,k)=(n,2)$.
  \item From (e) to (f): See \cite[Lemma 2.3]{PP}.
  \item From (f) to (g): Split the torus braid $(2n-1,-(3n-1))$ into two torus braids $(2n-1,-(n-1))$ and $(2n-1,-2n)$.
  \item From (g) to (h): Split the torus braid $(n,2)$ into two of torus braids $(n,1)$, and apply the lower left isotopy in Figure \ref{f:b}(b) with letting $(j,k)=(2n-1,n-1)$.
  \item From (h) to (i): Pull the lower $(n,1)$ down.
  \item From (i) to (j): Apply the lower left isotopy in Figure \ref{f:b}(b) with letting $(j,k)=(2n-1,n-1)$, and split the torus braid $(2n-1,-2n)$ into two torus braids $(2n-1,-1)$ and $-1_{2n-1}$.
  \item From (j) to (k): By using Lemma \ref{l:cjg}, delete the upper $(n,1)$ and attach it to the bottom. Also, by applying the right isotopy in Figure \ref{f:b}(d) with letting $(j,\ell)=(2n-1,n)$, divide the torus braid $(2n-1,-1)$ into two of $(n,-1)$.
  \item From (k) to (l): Pull the lower $(n,-1)$ down through the negative full twist $-1_{2n-1}$.
\end{itemize}
Two pairs of $(n,1)$ and $(n,-1)$ in the braid in Figure \ref{f:o}(l) can be canceled. It is clear that the closure of the resulting braid is the torus knot $T(2n-1,-(3n-2))$. Thus
$T(3n+1,n,3n-1,-1)=T(3n-2,-2n+1)$.

\end{document}